\documentclass[11pt]{amsart}
\usepackage{setspace, amsmath, amsthm, amssymb, amsfonts, amscd, epic, graphicx, ulem, dsfont}
\usepackage[T1]{fontenc}
\usepackage{multirow}
\usepackage{bbm}
\usepackage{enumerate}

\makeatletter \@namedef{subjclassname@2010}{
  \textup{2010} Mathematics Subject Classification}
\makeatother

\newtheorem{thm}{Theorem}[section]
\newtheorem{cor}[thm]{Corollary}
\newtheorem{lem}[thm]{Lemma}
\newtheorem{pro}[thm]{Proposition}

\theoremstyle{remark}
\newtheorem*{rema}{Remark}

\newtheorem{exa}[thm]{\textbf{Example}}

\theoremstyle{definition}
\newtheorem*{defn}{Definition}

\begin{document}

\title{Generalizations of Reid Inequality}
\author[S. DEHIMI and M. H. MORTAD]{Souheyb Dehimi and Mohammed Hichem Mortad$^*$}

\thanks{* Corresponding author.}
\date{}


\address{Department of
Mathematics, University of Oran 1, Ahmed Ben Bella, B.P. 1524, El
Menouar, Oran 31000, Algeria.\newline {\bf Mailing address}:
\newline Pr Mohammed Hichem Mortad \newline BP 7085 Seddikia Oran
\newline 31013 \newline Algeria}

\email{sohayb20091@gmail.com} \email{mhmortad@gmail.com,
mortad@univ-oran.dz.}

\begin{abstract}In this paper, we improve the famous Reid Inequality related
to linear operators. Some monotony results for positive operators
are also established with a different approach from what is known in
the existing literature. Lastly, Reid and Halmos-Reid inequalities
are extended to unbounded operators.
\end{abstract}

\maketitle

\section{Introduction}
First, assume that readers are familiar with notions and result on
$B(H)$. We do recall a few definitions and results though:
\begin{enumerate}
  \item Let $A\in B(H)$. We say that $A$ is positive (we then write $A\geq 0$) if
  \[<Ax,x>\geq0,~\forall x\in H.\]
  \item For every positive operator $A\in B(H)$, there is a unique positive $B\in B(H)$
  such that $B^2=A$. We call $B$ the positive square root of $A$.
  \item The absolute value of $A\in B(H)$ is defined to be the
  (unique) positive square root of the positive operator $A^*A$.
  We denote it by $|A|$.
  \item If $A\geq B\geq 0$, then $\sqrt A\geq \sqrt B$ (a particular
  case of the so-called Heinz Inequality).
  \item We say that $A\in B(H)$ is hyponormal if $AA^*\leq
  A^*A$. Equivalently, $\|A^*x\|\leq\|Ax\|$ for all $x\in H$.
  \item We also need the following lemma:
  \begin{lem}\label{Lemma MAIN LEMMA}(see e.g. \cite{Young-BOOK-HIlbert}) Let $H$ be a complex Hilbert
space. If $A,B\in B(H)$, then
\[\forall x\in H:~\|Ax\|\leq \|Bx\|\Longleftrightarrow \exists K\in B(H) \text{ contraction}:~A=KB.\]
\end{lem}
\end{enumerate}

The inequality of Reid which first appeared in
\cite{Reid-Inequality} is recalled next:

\begin{thm}\label{Redi classic THM}
Let $A,K\in B(H)$ be such that $A$ is positive and $AK$ is
self-adjoint. Then
\[|<AKx,x>|\leq ||K||<Ax,x>\]
for all $x\in H$.
\end{thm}

\begin{rema}
As shown in e.g. \cite{Lin-Reid-Furuta-inequality}, Reid Inequality
is equivalent to the operator monotony of the positive square root
on the set of positive operators.
\end{rema}

Halmos in \cite{halmos-book-1967} improved the inequality by
replacing $\|K\|$ by $r(K)$ where $r(K)$ is the usual spectral
radius. We shall call this the Halmos-Reid Inequality. Other
generalizations of Theorem \ref{Redi classic THM} are known in the
literature from which we only cite \cite{Kitanneh-REid-et al} and
\cite{Lin-Reid-Furuta-inequality}.

In an earlier version of this paper (see \cite{Mortad-Reid-Normal!
FIRST VERSION}), the corresponding author showed the following:

\begin{thm}\label{Main THM}
Let $A,K\in B(H)$ be such that $A$ is positive and $AK$ is normal.
Then
\[|<AKx,x>|\leq ||K||<Ax,x>\]
for all $x\in H$.
\end{thm}

Can we go up to hyponormal $AK$? In fact, the result is not true
even when $AK$ is quasinormal (and so we cannot go up to subnormal
either). The counterexample is given next:

\begin{exa}
Let $S$ be the shift operator on $\ell^2$. Setting $A=SS^*$, we see
that $A\geq0$. Now, take $K=S$ (and so $\|K\|=1$). It is clear that
$AK=SS^*S=S$ is quasinormal. If Reid Inequality held, then we would
have
\[|<Sx,x>|\leq <SS^*x,x>=\|S^*x\|^2\]
for each $x\in\ell^2$. This inequality clearly fails to hold for all
$x$. Indeed, taking $x=(2,1,0,0,\cdots)$, we see that
\[|<Sx,x>|=2\leq \|S^*x\|^2=1\]
which is absurd.
\end{exa}

The good news is that Reid Inequality can yet be improved as it
holds if $AK$ is co-hyponormal, that is, if $(AK)^*$ is hyponormal.
The proof, however, relies on the following result:

\begin{lem}\label{Kittaneh Lemma abos valu hypo}(\cite{Kitanneh-REid-et
al}, cf. Theorem \ref{hklmmm}) Let $A\in B(H)$ be hyponormal. Then
\[|<Ax,x>|\leq <|A|x,x>.\]
\end{lem}

This and some interesting consequences may be found in Section
\ref{fg}.

In Section \ref{klj} we treat Reid (and Halmos-Reid) Inequality for
unbounded operators. Fortunately, the latter inequality does hold
for co-hyponormal unbounded operators as well. The proof is a little
more technical and so it seems appropriate to recall a couple of
definitions here (other notions on unbounded operators are assumed,
cf. \cite{SCHMUDG-book-2012}):

\begin{defn}\label{defn pos unbd schmud}
Let $T$ and $S$ be unbounded positive self-adjoint operators. We say
that $S\geq T$ if $D(S^{\frac{1}{2}})\subseteq D(T^{\frac{1}{2}})$ and $%
\left\Vert S^{\frac{1}{2}}x\right\Vert \geq \left\Vert T^{\frac{1}{2}%
}x\right\Vert$ for all $x\in D(S^{\frac{1}{2}}).$
\end{defn}

\begin{rema}
Heinz Inequality is valid for positive unbounded operators as well.
See e.g. \cite{SCHMUDG-book-2012}.
\end{rema}

Now, we recall the definition of an unbounded hyponormal operator.

\begin{defn}
A densely defined operator $T$ with domain $D(T)$ is called
hyponormal if
\[D(T)\subset D(T^*)\text{ and } \|T^*x\|\leq\|Tx\|,~\forall x\in D(T).\]
\end{defn}

\section{Main Results: The Bounded Case}\label{fg}

\begin{thm}\label{Main THM}
Let $A,K\in B(H)$ be such that $A$ is positive and $(AK)^*$ is
hyponormal. Then
\[|<AKx,x>|\leq ||K||<Ax,x>\]
for all $x\in H$.
\end{thm}

\begin{proof}
  The inequality is evident when $K=0$. So, assume that
$K\neq0$. It is then clear that $\frac{K}{\|K\|}$ satisfies
\[KK^*\leq \|K\|^2I.\]
Hence
\[|(AK)^*|^2=AKK^*A\leq \|K\|^2A^2\]
or simply $|(AK)^*|\leq \|K\|A$ after passing to square roots.

Now, for all $x\in H$
\[|<AKx,x>|=|<x,(AK)^*x>|=|\overline{<(AK)^*x,x>}|=|<(AK)^*x,x>|.\]

Since $(AK)^*$ is hyponormal, Lemma \ref{Kittaneh Lemma abos valu
hypo} combined with $|(AK)^*|\leq \|K\|A$ give
\[|<AKx,x>|=|<(AK)^*x,x>|\leq |<|(AK)^*|x,x>|\leq \|K\|<Ax,x>\]
and this marks the end of the proof.
\end{proof}

\begin{thm}\label{theorem squrae root first}
Theorem \ref{Main THM} can be deduced from the operator monotony of
the positive square root on the set of positive operators, and vice
versa.
\end{thm}

\begin{proof}
Let $A,B\in B(H)$.
\begin{enumerate}
  \item Assume that $0\leq A\leq B$. Let $x\in H$. Since $A\leq B$, we easily see that:
\[\|\sqrt A x\|^2\leq \|\sqrt B x\|^2.\]

So, by Lemma \ref{Lemma MAIN LEMMA}, we know that $\sqrt A=K\sqrt B$
for some contraction $K\in B(H)$. Since $\sqrt A$ is self-adjoint,
it follows that $K\sqrt B$ too is self-adjoint (hence
co-hyponormal!). As $\sqrt B\geq0$, then by Reid Inequality (Theorem
\ref{Main THM}) we obtain:
\[<\sqrt Ax,x>=<\sqrt BK^*x,x>\leq <\sqrt Bx,x>\]
or
\[\sqrt A\leq \sqrt B,\]
as required.
  \item The other implication, which uses the fact that the square
  root is increasing, has already been presented in the proof of
  Theorem \ref{Main THM}.
\end{enumerate}
\end{proof}

The idea of the first part of the proof of the preceding result may
be used to produce new proofs of other known results on monotony.

\begin{thm}\label{theorem 2}
Let $A,B\in B(H)$. If $0\leq A\leq B$ and if $A$ is invertible, then
$B$ is invertible and $B^{-1}\leq A^{-1}$.
\end{thm}

\begin{proof}
As in the proof of Theorem \ref{theorem squrae root first}, we know
that $\sqrt A=K\sqrt B$ for some contraction $K\in B(H)$. Since
$\sqrt A$ is invertible (as $A$ is), it follows that $I=(\sqrt
A)^{-1}K\sqrt B$, i.e. the self-adjoint $\sqrt B$ is left invertible
and
  so $\sqrt B$ or simply $B$ is invertible (cf. \cite{Dehimi-Mortad-II}) and
\[(\sqrt B)^{-1}=(\sqrt A)^{-1}K=K^*(\sqrt A)^{-1}\]
by the self-adjointness of both $(\sqrt B)^{-1}$ and $(\sqrt
A)^{-1}$.

Let $x\in H$. Then (since $K^*$ too is a contraction)
\[<B^{-1}x,x>=\|(\sqrt B)^{-1}x\|^2=\|K^*(\sqrt A)^{-1}x\|^2\leq \|(\sqrt A)^{-1}x\|^2=<A^{-1}x,x>,\]
as needed.
\end{proof}

The following improvement of Lemma \ref{Lemma MAIN LEMMA} makes
proofs in case of commutativity very simple.

\begin{lem}\label{Lemma MAIN Stochel improvement LEMMA}Let $H$ be a complex Hilbert
space. If $A,B\in B(H)$ are self-adjoint and $BA\geq 0$, then
\[\forall x\in H:~\|Ax\|\leq \|Bx\|\Longleftrightarrow \exists K\in B(H) \text{ \textbf{positive} contraction}:~A=KB.\]
\end{lem}

\begin{proof}\hfill
\begin{enumerate}
  \item "$\Leftarrow$":  Let $x\in H$. Then
  \[0 \leq <KBx,Bx> = <Ax,Bx> = <BAx,x>,\]
  that is, $BA\geq 0$.
  \item "$\Rightarrow$": Since $BA \geq 0$, it follows that $BA$ is self-adjoint, i.e. $AB=BA$. As a
consequence, $\ker A$ reduces $A$ and $B$, and the restriction of
$A$ to $\ker A$ is the zero operator on $\ker A$. Hence, we can
assume that $A$ is injective. Therefore, because $\ker B \subset
\ker A =\{0\}$, we see that $B^{-1}$ is self-adjoint and densely
defined. Set $K_0=AB^{-1}$. Then $K_0$ is densely defined and

\[||K_0(Bx)|| = ||AB^{-1}Bx|| = ||Ax|| \leq||Bx||,  \forall x \in H,\]

\end{enumerate}
signifying that $K_0$ is a contraction with a unique contractive
extension $K$ to the whole $H$. Since
\[<K_0(Bx), Bx> = <Ax,Bx>=<BAx,x> \geq 0\]
for all $x \in H$, we see that $K$ is positive as well. Clearly
\[KBx= K_0(Bx) = Ax\]
for all $x\in H$, and this completes the proof.
\end{proof}

\begin{rema}
By scrutinizing the previous proof, we see that we can replace each
word "positive" by "self-adjoint" in the statement and in the proof.
Hence we also have:
\end{rema}

\begin{pro}
Let $A,B\in B(H)$ be self-adjoint such that $AB$ is self-adjoint
(that is, iff $AB=BA$). Then
\[\forall x\in H:~\|Ax\|\leq \|Bx\|\Longleftrightarrow \exists K\in B(H) \text{ \textbf{self-adjoint} contraction}:~A=KB.\]
\end{pro}

The next result is known. Its proof is a simple application of Lemma
\ref{Lemma MAIN Stochel improvement LEMMA}.

\begin{cor}\label{ppp} Let $A,B\in B(H)$ be positive and commuting. Then
\[0\leq A\leq B\Longrightarrow A^2\leq B^2.\]
\end{cor}

\begin{proof}Since $AB\geq 0$, we know by Lemma \ref{Lemma MAIN Stochel improvement
LEMMA} that $\sqrt A=K\sqrt B$ for some positive contraction $K\in
B(H)$ and $K\sqrt B=\sqrt B K$. Hence
\[A=K\sqrt B K\sqrt B=K^2B.\]
So for all $x\in H$:
\[\|Ax\|^2=\|K^2Bx\|^2\leq \|Bx\|^2\]
\text{ or merely }
\[<A^2x,x>=<Ax,Ax>=\|Ax\|^2\leq \|Bx\|^2=<B^2x,x>,\]
as required.
\end{proof}

\begin{rema}
The previous proof could even be shortened by directly using Reid
Inequality. As is presented, it can be given in courses which do not
cover Reid Inequality.
\end{rema}

By invoking the functional calculus of positive operators, we know
that we can define $A^{\alpha}$ for any (real) $\alpha>0$ (we may
allow $\alpha=0$) whenever $A\geq0$. We also know that if $B$ is
also positive and it commutes with $A$, then (using the spectral
theorem or else)
\[(AB)^{\alpha}=A^{\alpha}B^{\alpha}.\]

As a generalization of Corollary \ref{ppp}, we have

\begin{pro}\label{007} Let $A,B\in B(H)$ be positive and commuting. Then
\[0\leq A\leq B\Longrightarrow A^{\alpha}\leq B^{\alpha}\]
for any $\alpha\in (0,\infty)$.
\end{pro}

To prove it, we need the following perhaps known result:

\begin{lem}\label{45}Let $A\in B(H)$ be such that $A\geq0$. Then
\[\|A^{\alpha}\|=\|A\|^{\alpha}\]
for any $\alpha\in (0,\infty)$.
\end{lem}

\begin{proof}One of the ways of seeing this is to use the spectral radius theorem
and the fact that $A^{\alpha}$ is positive.
\end{proof}

Now, we give the proof of Proposition \ref{007}:

\begin{proof}Since $AB\geq 0$, we know by Lemma \ref{Lemma MAIN Stochel improvement
LEMMA} that $\sqrt A=K\sqrt B$ for some positive contraction $K\in
B(H)$ and $K\sqrt B=\sqrt B K$. Hence
\[A=K\sqrt B K\sqrt B=K^2B=BK^2.\]
Ergo
\[A^{\alpha}=(BK^2)^{\alpha}=B^{\alpha}K^{2\alpha}\text{ (because $K^2B=BK^2$)}.\]
Finally, for any $x\in H$, Reid Inequality (and a glance at Lemma
\ref{45}) allows us to write:
\[<A^{\alpha}x,x>=<B^{\alpha}K^{2\alpha}x,x>\leq \|K^{2\alpha}\|<B^{\alpha}x,x>=\|K\|^{2\alpha}<B^{\alpha}x,x>\leq <B^{\alpha}x,x>,\]
as desired.
\end{proof}

\section{Main Results: The Unbounded Case}\label{klj}

We start with the next practical result (probably known):

\begin{pro}\label{565}
Let $T$ be a closed hyponormal operator. Then
\[T^*T\geq TT^*.\]
\end{pro}

\begin{proof}First, since $T$ is closed, both $T^*T$ and $TT^*$ are
self-adjoint and positive (cf. \cite{SCHMUDG-book-2012}). As in the
bounded case, write $|T|=(T^*T)^{\frac{1}{2}}$. Then it is known
that
\[D(|T|)=D(T)\subseteq D(T^{\ast })=D(|T^{\ast }|).\]

Finally, for all $x\in D(T)$, we have
\[
\left\Vert |T^{\ast }|x\right\Vert =\left\Vert T^{\ast }x\right\Vert
\leq \left\Vert Tx\right\Vert =\left\Vert |T|x\right\Vert .
\]%

Therefore, according to Definition \ref{defn pos unbd schmud}, we
have $T^*T\geq TT^*$, as required.
\end{proof}

\begin{rema}
By Definition \ref{defn pos unbd schmud} above and some trivial
observations, $T^*T\geq TT^*$ clearly implies that $T$ is
hyponormal.
\end{rema}

Now, we prove the analogue of Lemma \ref{Kittaneh Lemma abos valu
hypo} for unbounded operators. The proof uses the Generalized
Cauchy-Schwarz Inequality for unbounded self-adjoint positive
operators.

\begin{thm}\label{hklmmm}
Let $T$ be a closed hyponormal operator. Then%
\[
|<Tx,x>|\leq <|T|x,x>\text{ for all }x\in D(T)\text{ .}
\]
\end{thm}

\begin{proof}
Let $T=U|T|$ be the polar decomposition of $T$ where $U$ is partial
isometry. Remember that (see e.g. \cite{Kat})
\[
|T^{\ast }|=U|T|U^{\ast }
\]%
and%
\[
U^{\ast }U|T|=|T|.
\]%
By Proposition \ref{565}, $TT^{\ast }\leq T^{\ast }T.$ Hence by
Heinz Inequality ("unbounded" version), we infer that $|T^*|\leq
|T|$.

Now, to show the required inequality, let $x\in D(T)$.  Then, we
have
\[
\begin{array}{lll}
|<Tx,x>|^{2} & = & |<U|T|x,x>|^{2} \\
& = & |<|T|x,U^{\ast }x>|^{2} \\
& \leq  & <|T|x,x><|T|U^{\ast }x,U^{\ast }x> \\
& = & <|T|x,x><U|T|U^{\ast }x,x> \\
& = & <|T|x,x><|T^*|x,x> \\
& \leq  & <|T|x,x><|T|x,x> \text{(as $T$ is hyponormal)}.%
\end{array}%
\]

Accordingly,
\[
|<Tx,x>|\leq <|T|x,x>,
\]
as required.
\end{proof}

We are ready to give the "unbounded" analogue of Theorem \ref{Main
THM}.

\begin{thm}\label{Reid co-hyponormal ultimate generalization}
Let $K$ be a bounded operator and let $A$ be a non-necessarily
bounded self-adjoint positive
operator such that $(AK)^{\ast }$ is hyponormal. Then%
\[
|<AKx,x>|\leq \left\Vert K\right\Vert <Ax,x>
\]%
for all $x\in D(A)$.
\end{thm}

The proof of the preceding theorem is based on the following simple
lemma.

\begin{lem}\label{KK*leq I implies AK(AK)* leq I Lemma UNB}
Let $A$ be a non necessarily bounded self-adjoint operator. Let
$K\in B(H)$ be such that $KK^*\leq \alpha I$ for some $\alpha>0$.
Then
\[AK(AK)^*\leq \alpha A^2\]
whenever $AK$ is densely defined.
\end{lem}

\begin{proof}
Since $AK$ is closed as $A$ is and $K$ is bounded, we clearly have
\[
\left\vert \left( AK\right) ^{\ast }\right\vert ^{2}=AK\left(
AK\right) ^{\ast }\leq  \alpha A^{2}.
\]%
To show the required inequality, notice first that
\[
D\left(  \alpha A\right) =D\left( A\right) =D\left( K^{\ast
}A\right) \subseteq D\left( \left( AK\right) ^{\ast }\right)
=D\left( \left\vert \left( AK\right) ^{\ast }\right\vert \right).
\]%
According to Definition \ref{565}, we need only check that
$\left\Vert \left\vert \left( AK\right) ^{\ast }\right\vert
x\right\Vert \leq \sqrt{\alpha} \left\Vert Ax\right\Vert $ for all
$x\in D\left( A\right)$. So, let $x\in D\left( A\right)$. Then
\[K^{\ast }Ax=%
\left( AK\right) ^{\ast }x.\] Hence for all $x\in D(A)$:
\[
\begin{array}{lll}
\left\Vert \left\vert \left( AK\right) ^{\ast }\right\vert
x\right\Vert ^{2}
& = & \left\Vert \left( AK\right) ^{\ast }x\right\Vert ^{2} \\
& = & \left\Vert K^{\ast }Ax\right\Vert ^{2} \\
& \leq & \left\Vert K\right\Vert ^{2}\left\Vert Ax\right\Vert ^{2}
\end{array}
\]
or simply $\left\vert \left( AK\right) ^{\ast }\right\vert \leq
\left\Vert K\right\Vert A$ (where obviously $\alpha=\|K\|^2$).
\end{proof}

We are ready to give a proof of Theorem \ref{Reid co-hyponormal
ultimate generalization}:

\begin{proof}
First, since $(AK)^{\ast }$ is hyponormal, we clearly have
\[D(A)\subseteq D((AK)^*)\subset D((AK)^{**})=D(AK)\]
because $AK$ is also closed.

Now, the inequality is evident when $K=0$. So, assume that $K\neq
0$. It is then
clear that $\frac{K}{||K||}\neq 0$ satisfies%
\[
KK^{\ast }\leq \left\Vert K\right\Vert ^{2}I.
\]%
Lemma \ref{KK*leq I implies AK(AK)* leq I Lemma UNB} then yields
\[\left\vert \left( AK\right) ^{\ast }\right\vert ^{2}=AK\left( AK\right)
^{\ast }\leq \left\Vert K\right\Vert ^{2}A^{2}.\]

Therefore, for all $x\in
D\left( A\right) $%
\[
\begin{array}{lll}
|<AKx,x>| & = & |<x,\left( AK\right) ^{\ast }x>| \\
& = & \overline{|<\left( AK\right) ^{\ast }x,x>|} \\
& = & |<\left( AK\right) ^{\ast }x,x>| \\
& \leq  & <\left\vert \left( AK\right) ^{\ast }\right\vert x,x>|\text{ \ (}%
\left( AK\right) ^{\ast }\text{ is hyponormal).} \\
& \leq  & \left\Vert K\right\Vert <Ax,x>%
\end{array}%
\]
and this marks the end of the proof.
\end{proof}

In the end, we give the generalization of Halmos-Reid Inequality.
The proof, which uses a standard argument (cf.
\cite{halmos-book-1967} or \cite{Lin-Dragomir}), is more technical
in the unbounded case.

\begin{thm}
Let $K$ be a bounded operator, and let $A$ be an unbounded
self-adjoint positive operator such that $K^{\ast }A\subseteq AK,$
then
\[
|<AKx,x>|\leq r\left( K\right) <Ax,x>\text{ \ for all }x\in D\left(
A\right) .
\]%
where $r\left( K\right) $ denotes the spectral radius of $K.$
\end{thm}

\begin{proof}
Since $K^{\ast }A\subseteq AK$, we can get by induction that
\[K^{\ast n}A\subseteq AK^{n} \text{ or merely } (AK^{n})^{*}\subseteq AK^{n}
\]%
for all $n$. Hence $\left( AK^{n}\right) ^{\ast }$ is hyponormal. We
also observe that as $A^{\frac{1}{2}}A^{\frac{1}{2}}=A$, then
$D(A)\subset D(A^{\frac{1}{2}})$.

Now, let us prove the following key result:
\[
|<AKx,x>|\leq
<AK^{2n}x,x>^{\frac{1}{2^{n}}}<Ax,x>^{\frac{2^{n}-1}{2^{n}}}
\]
We use a proof by induction. For $n=1$, we have for all $x\in D(A)$:
\[
\begin{array}{lll}
|<AKx,x>| & \leq & |<A^{\frac{1}{2}}A^{\frac{1}{2}}Kx,x>| \\
& = & |<A^{\frac{1}{2}}Kx,A^{\frac{1}{2}}x>|\text{ \ \ because }A^{\frac{1}{2%
}}Kx\in D(A^{\frac{1}{2}}) . \\
& \leq & \left\Vert A^{\frac{1}{2}}Kx\right\Vert \left\Vert A^{\frac{1}{2}%
}x\right\Vert \\
& = &
<A^{\frac{1}{2}}Kx,A^{\frac{1}{2}}Kx>^{\frac{1}{2}}<Ax,x>^{\frac{1}{2}}
\\
& = &
<A^{\frac{1}{2}}A^{\frac{1}{2}}Kx,Kx>^{\frac{1}{2}}<Ax,x>^{\frac{1}{2}}
\\
& = & <AKx,Kx>^{\frac{1}{2}}<Ax,x>^{\frac{1}{2}} \\
& = & <K^{\ast }AKx,x>^{\frac{1}{2}}<Ax,x>^{\frac{1}{2}} \\
& = & <AK^{2}x,x>^{\frac{1}{2}}<Ax,x>^{\frac{1}{2}}%
\end{array}%
\]%
Using a similar argument we can prove it for $n+1$ if it holds for
$n$.

Therefore, for all $n$ (and all $x\in D(A)$)
\[
|<AKx,x>|\leq
<AK^{2n}x,x>^{\frac{1}{2^{n}}}<Ax,x>^{\frac{2^{n}-1}{2^{n}}}.
\]%
Since $\left( AK^{2n}\right)^{\ast }$ is hyponormal, we might apply
Theorem \ref{Reid co-hyponormal ultimate generalization} to get
\[
\begin{array}{lll}
|<AKx,x>| & \leq  & <AK^{2n}x,x>^{\frac{1}{2^{n}}}<Ax,x>^{\frac{2^{n}-1}{%
2^{n}}} \\
& \leq  & \left\Vert K^{2n}\right\Vert ^{\frac{1}{2^{n}}}<Ax,x>^{\frac{1}{%
2^{n}}}<Ax,x>^{\frac{2^{n}-1}{2^{n}}} \\
& = & \left\Vert K^{2n}\right\Vert ^{\frac{1}{2^{n}}}<Ax,x>%
\end{array}%
\]%
whichever $n$. Passing to the limit finally yields
\[
|<AKx,x>|\leq r\left( K\right) <Ax,x>\text{ \ for all }x\in D\left(
A\right),
\]
as needed.
\end{proof}

\begin{cor}
Let $K$ be a bounded operator, and let $A$ be an unbounded
self-adjoint positive operator such that $AK$ is self-adjoint$,$
then
\[
|<AKx,x>|\leq r\left( K\right) <Ax,x>\text{ \ for all }x\in D\left(
A\right) .
\]
\end{cor}

\begin{proof}
The proof simply follows from
\[K^*A\subset (AK)^*=AK.\]
\end{proof}

\section{Conclusion}

The new bounded version of Reid Inequality is equivalent to a bunch
of other properties. Indeed, since we have shown that it is
equivalent to the fact that the square root is increasing on the set
of positive operators, we may call on
\cite{Lin-Reid-Furuta-inequality} to list other equivalent
conditions.

\section*{Acknowledgement}

We warmly thank Professor J. Stochel for communicating Lemma
\ref{Lemma MAIN Stochel improvement LEMMA} to the corresponding
author.

\end{document}